\documentclass[11pt]{amsart}

\usepackage[left=1in, right=1in, top=1.2in, bottom=1.2in]{geometry}
\usepackage{microtype, mathtools, mathrsfs, enumerate, dsfont, esvect, tikz}
\usepackage{verbatim}
\usepackage[linktocpage]{hyperref}
\hypersetup{
    colorlinks=true,        % false: boxed links; true: colored links
    linkcolor=red,          % color of internal links (change box color with linkbordercolor)
    citecolor=blue,         % color of links to bibliography
    filecolor=magenta,      % color of file links
    urlcolor=cyan           % color of external links
}

\usepackage{amssymb}
\usepackage{url}
\newtheorem{theorem}{Theorem}[section]
\newtheorem{proposition}[theorem]{Proposition}

\newtheorem{lemma}[theorem]{Lemma}

\newtheorem{corollary}[theorem]{Corollary}

\theoremstyle{definition}

\newtheorem{definition}[theorem]{Definition}

\theoremstyle{plain}
\numberwithin{equation}{theorem}

\theoremstyle{remark}

\newif\ifhascomments \hascommentstrue
\ifhascomments
  \newcommand{\dragos}[1]{{\color{red}[[\ensuremath{\bigstar\bigstar\bigstar} #1]]}}
  \newcommand{\matt}[1]{{\color{red}[[\ensuremath{\spadesuit\spadesuit\spadesuit} #1]]}}
\else
  \newcommand{\dragos}[1]{}
  \newcommand{\matt}[1]{}
\fi

\begin{document}

\title[On free subgroups in division rings]{On free subgroups in division rings}

\author{Jason P. Bell}
\address{Department of Pure Mathematics\\
University of Waterloo\\
Waterloo, ON N2L 3G1\\
Canada}
\email{jpbell@uwaterloo.ca}

\author{Jairo Gon\c calves}
\address{
Department of Mathematics \\
University of Sa\~ o Paulo\\
 Sa\~ o Paulo, SP, 05508-090\\
  Brazil}
\email{jz.goncalves@usp.br}

\begin{abstract} Let $K$ be a field and let $\sigma$ be an automorphism and let $\delta$ be a $\sigma$-derivation of $K$.  Then we show that the multiplicative group of nonzero elements of the division ring $D=K(x;\sigma,\delta)$ contains a free non-cyclic subgroup unless $D$ is commutative, answering a special case of a conjecture of Lichtman.  As an application, we show that division algebras formed by taking the Goldie ring of quotients of group algebras of torsion-free non-abelian solvable-by-finite groups always contain free non-cyclic subgroups. \end{abstract}

\subjclass[2010]{12E15, 16K40, 20E05}
%16D60 %Global ground fields
%16A20, %Linear algebraic groups over arbitrary fields
%16A32. %Iteration problems

\keywords{Division rings, free groups, group algebras, solvable-by-finite groups, Ore extensions}

\thanks{}

\maketitle

\section{Introduction}
Let $D$ be a division ring and let $D^{*}=D \setminus \{0\}$ the multiplicative group of $D$. The existence of free subobjects in $D$ and $D^*$ has been studied in several different contexts (see, for example, \cite{ML1, ML2, ML3, ML4, ML5, Lor86, Li2, Li3, FGF18}).  The most important of these are the study of free multiplicative semigroups of rank two in $D$, free non-cyclic subgroups of $D^*$, and free subalgebras of rank two in $D$. The guiding philosophy in all of these conjectures is that division rings that are not locally finite-dimensional over their centres should be very large and unwieldy and have big subgroups and big subalgebras. The question of the existence of free non-cyclic subgroups in the nonzero elements of a division ring $D$, goes back to a work of Lichtman \cite{Li}, who conjectured that if $D$ is a division ring that is not a field then $D^*$ contains a free group of rank two. This is known in several cases, including the case when $D$ is finite-dimensional over its centre, where it can be deduced from the Tits alternative for linear groups \cite{Li2, Gon}.  Beyond this case, Lichtman's conjecture has been verified for several special cases of algebras, including division rings whose multiplicative group contains a non-abelian nilpotent-by-finite subgroup \cite{Li2}, division algebras of polycyclic-by-finite group algebras \cite{Li3}, quotient division rings of enveloping algebras of finite-dimensional or solvable non-abelian Lie algebras over a field of characteristic zero \cite{Li4}. In terms of general results, arguably the best result concerning the existence of free subgroups in division rings is a result of Chiba's \cite{Chi}, who shows that if $D$ is a division ring with uncountable centre then $D^*$ contains a free group.  In this note, we prove Lichtman's conjecture in two cases that apply to a broad class of division rings: Ore extensions of fields and group algebras of torsion-free solvable-by-finite groups.

We recall that if $K$ is a field and $\sigma$ is an automorphism of $K$ and $\delta:K\to K$ is a $\sigma$-derivation (that is, $\delta(ab)=\delta(a)b+\sigma(a)\delta(b)$ for $a,b\in K$)
then we can form a skew polynomial ring $K[x;\sigma,\delta]$, which is, as a set, just the polynomial ring $K[x]$ but with ``skewed'' multiplication given by the rule that multiplication of elements of $K$ remains the same, but for $\alpha\in K$ and $x$ the indeterminate we have the rule
$x\cdot \alpha = \sigma(\alpha) x+\delta(\alpha)$. This ring is a noetherian domain and in a semiprime noetherian ring $R$ one can invert the regular elements (non-zero divisors of $R$) to obtain a quotient division ring ${\rm Frac}(R)$, which is semiprime Artinian and which we think of as being a ``noncommutative field of fractions.''  In the case when $R=K[x;\sigma,\delta]$, we let $K(x;\sigma,\delta)$ denote ${\rm Frac}(R)$, and this quotient ring is a division ring that is not a field unless $\sigma$ is the identity and $\delta=0$.  In the case that $\delta=0$ we simply write $K[x;\sigma]$ for $K[x;\sigma,\delta]$ and call this an Ore extension of \emph{automorphism type}.  In the case when $\sigma$ is the identity, $\delta$ is just an ordinary derivation and we write $K[x;\delta]$ for $K[x;\sigma,\delta]$ and call this an Ore extension of \emph{derivation type}.  Our main result is the following theorem. 

\begin{theorem}
\label{thm: main1}
Let $K$ be a field, let $\sigma$ be an automorphism of $K$ and let $\delta$ be a $\sigma$-derivation of $K$.  Then if  the division ring $K(x;\sigma,\delta)$ is not commutative then the group of units of $K(x;\sigma,\delta)$ contains a free non-cyclic subgroup.
\end{theorem}
In fact, in the pure automorphism case the result we prove is somewhat stronger than Theorem \ref{thm: main1} (see the statement of Theorem \ref{thm: main2}).  The stronger form is useful in terms of the strategy of proof, because we show there are ``lifting'' techniques which allow one to pass to simple homomorphic images of the rings we consider.
 
Division rings of the form $K(x;\sigma,\delta)$ have been considered in this context (see, for example, \cite{BR12, Lor86, FGS, FGF18}), but results of the form given in Theorem \ref{thm: main1}, when $K$ is countable, have only been proved in a few instances where one has an explicit form of the automorphism or some other constraints.  Once one moves beyond division rings that are finite-dimensional over their centres, where Lichtman's conjecture is known to hold \cite{Li2, Gon}, the first natural class of division rings to consider are precisely those formed by Ore extensions of fields, since they in some sense provide the first examples of division rings that are infinite-dimensional over their centres that are not locally finite-dimensional. 

The main issue that arises in this context is that one quickly encounters thorny dynamical problems involving birational maps when studying the existence of free subgroups in this setting.  We are able to prove this result by first reducing it to the case of algebras of the form $K(x;\sigma)$ with $K$ a finitely generated extension of the prime field.  We then show that one can further reduce to the case when $K$ has positive characteristic.  Our main tool in this setting is the following deep result of Hrushovski's, which the authors learned of through an application of this result by Amerik \cite{Ame}.
\begin{theorem} \emph{(}\cite[Corollary 1.2]{Hr}\emph{)} Let $U$ be an affine variety over a finite field $\mathbb{F}_q$ and let $S \subseteq U^2$ be an irreducible subvariety over $\bar{\mathbb{F}}_q$. Assume that the two projections of $S$ to $U$ are dominant. Let $\phi_q$ denote the $q$-Frobenius map. Then for any proper subvariety $W$ of $U$, for sufficiently large $m$, there exists $x\in U(\bar{\mathbb{F}}_q)$ with $(x,\phi_q^m(x))\in S$ and $x\not\in W$.
\label{thm:Hrush}
\end{theorem}
Using Hrushovski's result, we can reduce the positive characteristic case to the case of an algebra satisfying a polynomial identity, which we are able to handle via the Tits alternative.  A consequence of Theorem \ref{thm: main1} is the following result concerning group algebras.
\begin{corollary} Let $G$ be a non-abelian torsion-free solvable-by-finite group and let $k$ be a field. Then the the quotient division algebra of the group algebra $k[G]$ contains a free non-cyclic multiplicative subgroup.
\label{corollary:main}
\end{corollary}
The corollary follows from Theorem \ref{thm: main1} in a straightforward manner.  First, a non-abelian solvable group contains a non-abelian metabelian subgroup and so we can quickly reduce to the metabelian case.  It can be shown that non-abelian metabelian groups either contain a non-abelian abelian-by-finite subgroup or a semidirect product of an abelian group with an infinite cyclic group; the second case can be handled with Theorem \ref{thm: main1}, while the first case follows from work of Passman and Isaacs along with earlier work of Lichtman using the Tits alternative. 

The outline of the paper is as follows.  In \S\ref{sec:Pre} we give the basic facts involving regular extensions and introduce the notion of pro-unipotence, which will be a key part of the argument in providing a framework via which we can lift free subgroups.  In \S\ref{sec:pc} and \S\ref{sec:zc} we will prove the automorphism case of Theorem \ref{thm: main1} in the case when $K$ is finitely generated over a prime subfield of positive characteristic and zero characteristic respectively.  We need the result of Hrushovski mentioned above to handle the positive characteristic case; the characteristic zero case is handled via reduction to the positive characteristic case.  In \S\ref{sec:main} we prove Theorem \ref{thm: main1} and finally in \S\ref{sec:app} we prove Corollary \ref{corollary:main}.
\section{Preliminaries}
\label{sec:Pre}
In this section, we give some of the basic concepts we will need in proving Theorem \ref{thm: main1}.
We start with the notion of regular extensions.  Let $K$ be a field extension of a field $k$.  We recall that $K$ is a regular extension of $k$ if $K\otimes_k F$ is an integral domain for every field extension $F$ of $k$.  This in turn is equivalent to $k$ being algebraically closed in $K$ and $K$ having a separating transcendence base (i.e., $K$ has a subfield $L$ that is a purely transcendental extension of $k$ and such that $K$ is separable over $L$). In characteristic zero, having a separating transcendence base is automatic; in positive characteristic, this simply means that there is a purely transcendental extension $L$ of $k$ inside of $K$ such that $K$ is a separable algebraic extension of $L$.
The main significance of regularity is given by the following result.
\begin{proposition} Let $C$ be a finitely generated commutative $\mathbb{Z}$-algebra and let $R$ be a finitely generated commutative $C$-algebra with $C\subseteq R$ and suppose that $C$ and $R$ are integral domains.  If ${\rm Frac}(R)$ is a regular extension of ${\rm Frac}(C)$ then there is a nonzero $c\in C$ such that $R/PR$ is an integral domain for every prime $P$ in ${\rm Spec}(C)$ such that $c\not\in P$.
\label{prop:invert}
\end{proposition}
\begin{proof} Let $X={\rm Spec}(R)$ and let $Y={\rm Spec}(C)$. Then the inclusion $C\to R$ and the fact that $R$ is a finitely presented $C$-algebra gives a morphism of finite presentation $f: X\to Y$.  Let $S=C\setminus \{0\}$.  Then 
$S^{-1}R$ is a ${\rm Frac}(C)$-algebra with field of fractions ${\rm Frac}(R)$ and since ${\rm Frac}(R)$ is a regular extension of ${\rm Frac}(C)$, we have
$$S^{-1}R\otimes_{{\rm Frac}(C)}\overline{{\rm Frac}(C)}$$ is an integral domain.  This means that the generic fiber of $f$ is geometrically irreducible, and so by \cite[Theorem 9.7.7]{GD} there is a non-empty open subset $U$ of $Y$ such that for points $P\in U$ we have $R\otimes_C C/P = R/PR$ is an integral domain.  Then since $U$ is open, its complement is a closed subset and thus contained in a closed set of the form $\{Q\colon c\in Q\}$ for some $c$. It follows that $U$ contains the open set of primes $P$ such that $c\not\in P$ and so we obtain the result.
\end{proof}
\begin{definition}
Let $R$ be a commutative noetherian integral domain with an automorphism $\sigma$.  We say that an element $u$ of the units group of ${\rm Frac}(R[x;\sigma])$ is \emph{pro-unipotent} with respect to $R$ if it is of the form $(1+a)(1+b)^{-1}$ with $a,b\in xR[x;\sigma]$.  In the case that $R$ is a field we will call elements that are pro-unipotent with respect to $R$ simply pro-unipotent.
\end{definition}
We recall that we can complete $R[x;\sigma]$ at the prime ideal $(x)$ to obtain a skew power series ring $R[[x;\sigma]]$. The justification for the term pro-unipotent comes from the fact that if we let $I_n=x^n R[[x;\sigma]]$, then regarding a pro-unipotent element $u=(1+a)(1+b)^{-1}$ (with respect to $R$) as an element of $R[[x;\sigma]]$ (by expanding $(1+b)^{-1}$ as a skew power series and left multiplying by $(1+a)$, we see that it has the property that $u-1$ is nilpotent mod $I_n$ and so the image of $u$ is a unipotent element of $R[[x;\sigma]]/I_n$ for every $n\ge 1$. We remark that in the case when $R$ is a field it is straightforward to show that the pro-unipotent elements form a normal subgroup of ${\rm Frac}(R[x;\sigma])^*$. The significance for pro-unipotence is that pro-unipotent elements have well behaved lifting properties that allow us to lift free subgroups from homomorphic images. We make this precise with the following lemma.

\begin{lemma} Let $R$ and $S$ be commutative noetherian integral domains, let $\sigma$ and $\tau$ be automorphisms of $R$ and $S$ respectively, and suppose that there is a surjective homomorphism $\phi: R[x;\sigma]\to S[x;\tau]$ of graded algebras; that is, $\phi(Rx^n)=Sx^n$ for $n\in \mathbb{N}$.  Then if there exist elements $u,v$ in ${\rm Frac}(S[x;\tau])$ that are pro-unipotent with respect to $S$ that generate a free non-cyclic subgroup of the units group of ${\rm Frac}(S[x;\tau])$ then there exist $u',v'$ in ${\rm Frac}(R[x;\sigma])$, pro-unipotent with respect to $R$, that generate a free non-cyclic subgroup of ${\rm Frac}(R[x; \sigma])^*$.  
\label{lem:reduce}
\end{lemma}
\begin{proof} Since $\phi$ is a degree-preserving homomorphism, $\phi$ extends to a surjective homomorphism 
$\widehat{\phi} : R[[x;\sigma]]\to S[[t;\tau]]$.
By assumption, there exist $a,b,c,d\in xS[x;\tau]$ such that $u:=(1+a)(1+b)^{-1}$ and $v:=(1+c)(1+d)^{-1}$ are units in ${\rm Frac}(S[x;\tau])$ and such that they generate a free non-cyclic subgroup of the units group of $S[[x;\tau]]$.  Now $\phi$ is surjective and so we may pick $a',b',c',d'\in xR[x;\sigma]$ such that $\phi(a')=a,\phi(b')=b,\phi(c')=c,\phi(d')$.  Then
$u':=(1+a')(1+b')^{-1}, v':=(1+c')(1+d')^{-1}\in {\rm Frac}(R[x;\sigma])^*$ and regarding $u,v,u',v'$ as skew power series we have $\widehat{\phi}(u')=u$ and $\widehat{\phi}(v')=v$. Then since $u$ and $v$ generate a free group, we see that $u'$ and $v'$ must too, since a non-trivial relation involving $u'$ and $v'$ in ${\rm Frac}(R[x;\sigma])^*$ would give a relation in ${\rm Frac}(R[[x;\sigma]])^*$ and this would then give a corresponding non-trivial relation involving $u$ and $v$ via the map $\widehat{\phi}$.  \end{proof}

Our goal is to use this lemma as a reduction tool in proving the following strengthening of Theorem \ref{thm: main1}.
\begin{theorem}
\label{thm: main2}
Let $K$ be a field and let $\sigma$ be an automorphism of $K$ that is not the identity.  Then there exist pro-unipotent elements of ${\rm Frac}(K[x;\sigma])$ that generate a non-cyclic multiplicative free group.
\end{theorem}
In the next section, we shall begin the proof of this result using Lemma \ref{lem:reduce} to provide a series of reductions.

 \section{Theorem \ref{thm: main2} in the case that $K$ is finitely generated over a finite field}
\label{sec:pc}
In this section we prove the following result, which deals with Theorem \ref{thm: main2} in the case when $K$ is finitely generated over a finite field.  This is really the key step in our analysis and we divide the proof into the cases when the automorphism $\sigma$ has finite order and when it has infinite order; the infinite order case reduces to the finite order case.

The finite-order case is relatively straightforward and we prove a more general result.  %We recall some basic facts about rings with automorphisms.  Given a noetherian ring $R$ with an automorphism $\sigma$, we say that a proper $\sigma$-invariant ideal $I$ of $R$ is $\sigma$-\emph{prime} if whenever $J$ and $L$ are $\sigma$-invariant ideals containing $I$ such that $JL\subseteq I$ we must have $J=I$ or $L=I$.
\begin{proposition} \label{prop:PI} Let $K$ be a field and let $\sigma$ be a non-identity automorphism of $K$ of finite order. Then there are pro-unipotent $u,v$ in ${\rm Frac}(K[x;\sigma])$ that generate a free non-cyclic multiplicative subgroup of the units group. \end{proposition}
\begin{proof}
Since $\sigma$ has finite order, $D={\rm Frac}(K[x;\sigma])$ is finite-dimensional over its centre.  In particular, $D^*$ is a matrix group.  Notice that the set of pro-unipotent elements forms a normal subgroup of $D^*$.  Then since $N$ is non-abelian, we get the result in this case \cite{Li2, Gon}.
%Let $\mathcal{S}$ be a finite set of generators for $K$ as an extension of $k$.  We let $R$ denote the finitely generated $k$-algebra generated by $$\bigcup_{i=0}^{m-1}\sigma^i(\mathcal{S}).$$ Then $R$ is $\sigma$-invariant.  For each maximal ideal $P$ of $R$ we let $I(P)=\cap_{i=1}^m \sigma^i(P)$, which is a $\sigma$-invariant $\sigma$-prime ideal of $R$.  By the Nullstellensatz, $R/P$ is a finite ring and so $R/I(P)$ is also a finite ring for each maximal ideal $P$ of $R$.  Since $R$ is a Jacobson ring we see that $$\bigcap_P I(P)=(0)$$ as $P$ ranges over the maximal ideals of $R$.  Then since $\sigma$ is not the identity, there is some $I(P)$ such that $\sigma$ does not induce the identity automorphism of $R/I(P)$.  
 %Then we have a surjective homomorphism from $R[x;\sigma]\to (R/I(P))[x;\sigma]$.  By Lemma \ref{lem:fin} we have that the quotient division ring of $ (R/I(P))[x;\sigma]$ contains two pro-unipotent elements that generate a free multiplicative subgroup of the units group.  By Lemma \ref{lem:reduce} we have that ${\rm Frac}(R[x;\sigma])$ contains two pro-unipotent elements that generate a free subgroup of the units group.  The result follows in this case and we obtain the claimed result.
\end{proof}

 \begin{proposition} Let $k$ be a finite field and let $K$ be a finitely generated extension of $k$ and let $\sigma$ be a non-identity $k$-algebra automorphism of $K$.  Then there exist pro-unipotent elements $u$ and $v$ of $K(x;\sigma)$ such that $u$ and $v$ generate a free non-cyclic subgroup of the units group of $K(x;\sigma)$.
 \label{prop:pc}
 \end{proposition}
 \begin{proof}
 The case when $\sigma$ has finite order follows from Proposition \ref{prop:PI}. Thus we may assume that $\sigma$ has infinite order.  By enlarging $k$ if necessary and replacing $\sigma$ by $\sigma^m$ for some $m\ge 1$, we may work instead with $K(x^m;\sigma^m)$ and assume that $k$ is algebraically closed in $K$ and that $\sigma$ restricts to the identity on $k$.  Pick a transcendence base $t_1,\ldots ,t_d$ for $K/k$ and let $K'$ denote the separable closure of $k(t_1,\ldots ,t_d)$ inside $K$.  Then $K'$ is an extension of $k$ and by construction $k$ is algebraically closed in $K'$ and $K'$ has a separating transcendence base.  Thus $K'$ is a regular extension of $k$.  Then since the Frobenius map is surjective on $k$ and since $K$ is a finite extension of $K'$, there is some $s\ge 1$ such that $L:=K^{p^s}:=\{a^{p^s}\colon a\in K\}$ is contained in $K'$ and contains $k$. Then $L$ is a regular extension of $k$ as it is contained in $K'$.  Moreover, $\sigma$ restricts to a non-identity automorphism of $L$ since $p^s$-th roots of elements are unique in $K$.  Then $K'$ is a regular extension of $k$.  

Since $K'/k$ is regular, it follows that 
$L:=K'\otimes_k \bar{k}$ is a finitely generated field extension of $\bar{k}$ and thus $L$ is the function field of a normal projective variety $X$ defined over $\bar{k}$.  Moreover, the automorphism $\sigma$ induces a non-identity automorphism of $L$ and hence induces a birational automorphism $\phi$ of $X$ that is not the identity. Let $t_1,\ldots ,t_d$ be generators for $K'$ as a field extension of $k$.   Since $K'\subseteq L=\bar{k}(X)$, we can regard each $t_i$ as a rational map from $X$ to $\mathbb{P}^1_{\bar{k}}$.  Then 
since $\sigma$ is not equal to the identity, we have $\sigma(t_i)\neq t_i$ for some $i$ and by reindexing, we may assume that $i=1$.  Then $\sigma(t_1)-t_1$ is a nonzero map from $X$ to $\mathbb{P}^1_{\bar{k}}$, and we let $Y$ denote the zero locus of this map, which is a proper subvariety of $X$. Let $\phi$ be the birational map of $X$ induced by $\sigma$,  let $I$ denote the indeterminacy locus of $\phi$ and  let $Z$ denote the closed subvariety of $X\setminus I$ given by $\{x\in X\setminus I \colon \phi(x)=x\}$.  We let $W$ denote a codimension-$1$ subvariety of $X$ such that $t_1,\ldots ,t_d,\sigma(t_1)$ are regular on $X\setminus W$.

Then $X$, $Z$, $I$, $Y$, $W$, and $\phi$ are defined over $\bar{k}$, but since they are defined in terms of a finite amount of geometric data, they are in fact defined over a finite field $k'$ for some power extension $k'$ of $k$, and $\phi$ is defined over $k$.  Let $q$ denote the size of $k$, and so the $q$-power Frobenius acts as the identity on $k$. Then $Z$, $I$, $Y$, and $W$ all have a finite orbit under the action of the $q$-power Frobenius map.  We let $V$ denote the union of the $q$-Frobenius orbits of $Z$, $I$, $Y$, and $W$.  Then $V$ is a proper subvariety of $X$ and by Hrushovski's result (Theorem \ref{thm:Hrush}) we have that there is some $\bar{k}$-point $x\in X\setminus V$ such that $\phi(x)=F^j(x)$ for some $j$, where $F$ is the $q$-power Frobenius. Notice that since $\phi$ is defined over $k$ and so $\phi(F^j(x))= F^j(\phi(x))$ and since $x\not \in I$, we see that $F^j(x)\not \in V$ by definition of $V$ and so $F^j(x)\not\in I$ and so we have $\phi^2(x)= F^{2j}(x)$ and similarly we have $\phi^m(x)=F^{jm}(x)$.  Since $x$ is a $\bar{k}$-point and $\phi$ is defined over $k$ and since the orbit of $x$ avoids the indeterminacy locus of $\phi$, it has finite orbit under the automorphism $F^j$ and so there is some $m$ such that $\phi^m(x)=x$.  Let $S=\{x_1,\ldots ,x_d\}$ denote the orbit of $x$ under $\phi$.  Then consider the semilocal ring $\mathcal{O}_{X,S}$, the collection of functions in $L$ that are regular at all points of $S$.  By construction $\mathcal{O}_{X,S}$ is invariant under $\sigma$ and $\mathcal{O}_{X,S}$ has a $\sigma$-invariant ideal $P_{X,S}$, consisting of functions that vanish on $S$.  Then
$\mathcal{O}_{X,S}/P_{X,S}\cong \bar{k}^d$.  
 We let $T$ denote the elements in $K'$ that are in $\mathcal{O}_{X,S}$. By construction, the functions $t_1,\ldots ,t_d,\sigma(t_1)$ are all regular on $S$ and so they are in $T$.  Thus $T$ has field of fractions equal to $K'$.  Furthermore, by construction $t_1-\sigma(t_1)\not\in P_{X,S}$ and so
 $\sigma$ induces a non-identity automorphism of $T/P_{X,S}\cap T$.  Notice that each element of $T$ is a rational function in $t_1,\ldots ,t_d$ that is regular at all points of $S$.  Moreover, modulo $P_{X,S}$, each element of $t_i$ maps to an element of $\bar{k}^d$ and so $B:=T/P_{X,S}\cap T$ is a finite reduced ring and $\sigma$ induces a non-identity automorphism of $B$.  Notice also that if $f\in T$ vanishes on some $x_i$ then since $K'$ is a regular extension of $k$, we have, as a rational map from $X$ to $\mathbb{P}^1_{\bar{k}}$, that $f$ is defined over $k$ and so $f\circ F^{n}(x_i)= F^n(f(x_i))=0$ and so $f$ vanishes on all of $S$.  Thus $P_{X,S}\cap T$ is a prime ideal and so 
 $B$ is a finite field.  
 Then by Proposition \ref{prop:PI}, there are pro-unipotent elements 
 $u,v\in {\rm Frac}(B[t;\sigma])$ that generate a free group of rank two.  Thus by Lemma \ref{lem:reduce}, there are pro-unipotent elements 
 $u',v' \in {\rm Frac}(T[t;\sigma])$ that generate a free group of rank two.  Then since ${\rm Frac}(T[t;\sigma])= {\rm Frac}(K[x;\sigma])$ we get the result in this case.
 \end{proof}

\section{Theorem \ref{thm: main2} in the case when $K$ is finitely generated over the rational numbers}
\label{sec:zc}
In this section we show that Theorem \ref{thm: main2} holds when $K$ is a finitely generated extension of $\mathbb{Q}$.  To do this, we show that we can create a $\sigma$-invariant local subalgebra of $K$ that surjects onto a finitely generated extension of a finite field; we then use the results of \S\ref{sec:pc} and Lemma \ref{lem:reduce} to obtain the result in this case.
\begin{lemma} Let $K$ be a finitely generated extension of $\mathbb{Q}$, let $\sigma$ be a field automorphism of $K$ that is not the identity automorphism, and let $k$ be the fixed field of $\sigma$.  If $K$ is a regular extension of $k$ then there is a $\sigma$-invariant local subring $S\subseteq K$ with maximal ideal $P$ such that:
\begin{enumerate}
\item $S/P$ is a finitely generated extension of a finite field;
\item $\sigma$ induces a non-identity automorphism $\bar{\sigma}$ of $S/P$.
\end{enumerate}
\label{lem:S}
\end{lemma}
\begin{proof}   
We pick a finitely generated subring $C$ of $k$ such that ${\rm Frac}(C)=k$.  We pick a transcendence base $t_1,\ldots ,t_d$ for $K$ over $k$ Then $K$ is a finite-dimensional $k(t_1,\ldots ,t_d)$-vector space and we let $1=a_1,\ldots ,a_m$ be a basis for this space.   Since $k$ is algebraically closed in $K$ and $\sigma$ is not the identity, we may pick our transcendence base in such a way that $\sigma(t_1)\neq t_1$.

Then let $S$ denote the finitely generated $C$-algebra generated by $\sigma^j(t_i)$ and $\sigma^j(a_{\ell})$ for $j=-1,0,1$, $i=1,\ldots ,d$, and $\ell=1,\ldots ,m$.  We let $R=C[t_1,\ldots ,t_d]$.  Then there is some nonzero $h\in R$ such that $S[1/h]$ is integral over $R[1/h]$ and 
$S[1/\sigma(h)]$ is integral over $\sigma(R)[1/\sigma(h)]$ and $S[1/\sigma^{-1}(h)]$ is integral over $\sigma^{-1}(R)[1/\sigma^{-1}(h)]$.  
By Proposition \ref{prop:invert}, we can invert a single element of $C$ if necessary and assume that every prime ideal $P\in {\rm Spec}(C)$ has the property that $PS$ is a prime ideal of $S$ and $PS[1/\sigma^j(h)]$ is a prime ideal of $S[1/\sigma^j(h)]$ for $j=-1,0,1$, and so we may enlarge $C$ if necessary and assume that this is the case.

Now since $K$ is a algebraic over $k(t_1,\ldots ,t_d)$ and since ${\rm Frac}(S)=K$ and ${\rm Frac}(C[t_1,\ldots, t_d])=k(t_1,\ldots ,t_d)$, we see that there exists a nonzero polynomial $q(t_1,\ldots ,t_d)\in C[t_1,\ldots ,t_d]$ such that
$q \sigma^j(t_i), q \sigma^j(a_{\ell}), q\sigma^j(h)^{-1}\in S$ for $j=-2,\ldots ,2$ and $i=1,\ldots ,d$, and $\ell=1,\ldots ,m$.  Now since the intersection of primes of the form $PS$ with $P$ a maximal ideal of $C$ has trivial intersection, we may pick a maximal ideal $P$ of $C$ such that $(\sigma(t_1)-t_1)hq\not\in PS$.  Now we claim that $\sigma$ restricts to an automorphism of the local ring $S_{PS}$.  To see this, observe that
by construction each generator $u$ of the finite set of generators we have given for $S$ has the property that $\sigma(u),\sigma^{-1}(u)\in S_{PS}$, since $q\not\in PS$.  Furthermore, we have $\sigma(1/h)$ and $\sigma^{-1}(1/h)$ are in $S_{PS}$ by construction and so $\sigma$ and $\sigma^{-1}$ map $S[1/h]$ into $S_{PS}$.  Since $\sigma$ and $\sigma^{-1}$ are $C$-linear we also have $\sigma(PS[1/h])\subseteq PS_{PS}$ and $\sigma^{-1}(PS[1/h])\subseteq PS_{PS}$.  Now we claim that $\sigma^{-1}(PS_{PS})\cap S[1/h] = PS[1/h]$.  To see this, suppose that this is not the case.  Then since $\sigma(PS[1/h])\subseteq PS_{PS}$, we must have $\sigma^{-1}(PS_{PS})\cap S[1/h] =Q$ for some $Q\in {\rm Spec}(S[1/h])$ that properly contains $PS$.  Then since $S[1/h]$ is integral over $R[1/h]=C[t_1,\ldots ,t_d][1/h]$, we see that 
$Q\cap R[1/h]$ properly contains $SP\cap R[1/h]=PR[1/h]$. Thus $Q\cap R[1/h]$ contains a polynomial $f(t_1,\ldots ,t_d)\in C[t_1,\ldots ,t_d]$ that is not in $P[t_1,\ldots ,t_d]$.  But then by construction $\sigma(f)\in PS_{PS}$.  But notice that $\sigma$ maps $C[t_1,\ldots ,t_d]$ into $S$ and so $\sigma(f)\in PS_{PS}\cap S[1/h] = PS[1/h]$.  But we showed $\sigma^{-1}(PS[1/h])\subseteq PS_{PS}$ and so $f\in PS_{PS}$.  But since $f\in S[1/h]$ we see that $f\in S[1/h]\cap PS_{PS}= PS[1/h]$. In fact, we have $f\in R[1/h]$ and so $f\in PS[1/h]\cap R[1/h]$.  But this intersection is equal to $PR[1/h]$, which follows from the fact that $P$ is a prime ideal of $C$ and $PR[1/h]$ is a prime ideal of $R[1/h]$ since $R$ is a polynomial ring over $C$ and so $PS[1/h] = (PR[1/h])S[1/h]$ and by integrality we then have $PS[1/h]\cap R[1/h]=PR[1/h]$.  This is a contradiction.   It follows that $\sigma^{-1}(PS_{PS})\cap S[1/h]=PS[1/h]$ and similarly we have
$\sigma(PS_{PS})\cap S[1/h]=PS[1/h]$  But this means that the injective map $\sigma : S[1/h]\to S_{PS}$ sends elements outside of $PS[1/h]$ to units of $S_{PS}$ and so $\sigma$ restricts to an injective homomorphism from $S_{PS}$ to itself; similarly, $\sigma^{-1} :S_{PS}\to S_{PS}$ and so $\sigma$ restricts to an automorphism of $S_{PS}$.  Now since $C$ is a finitely generated $\mathbb{Z}$-algebra, we see that for maximal ideals $P$ of $C$ that $C/P$ is a finite field and $L:=S_{PS}/PS\cong {\rm Frac}(S/PS)$.  By construction $\sigma$-induces a $C/P$-algebra automorphism of $L$; this automorphism is not the identity since $\sigma(t_1)-t_1\in S$ and is not in $PS$ by construction and so we obtain the result. 
\end{proof}
From this we obtain the following result.
\begin{proposition} Let $K$ be a finitely generated extension of $\mathbb{Q}$ and let $\sigma$ be an automorphism of $K$.  Then there exist pro-unipotent elements $u$ and $v$ of $K(x;\sigma)$ such that $u$ and $v$ generate a free non-cyclic subgroup of the units group of $K(x;\sigma)$.
\label{prop:zc}
\end{proposition}
\begin{proof} If $\sigma$ has finite order on $K$ then we get the result from Proposition \ref{prop:PI}. Thus we may assume that $\sigma$ has infinite order.  Let $k$ denote the fixed field of $K$ and let $k_0$ denote the algebraic closure of $k$ in $K$.  Then since $K$ is finitely generated over $k$, we have $k_0$ is finitely generated over $k$ \cite{Vam} and so $k_0$ is a finite extension of $k$.  Then there is some $m>1$ such that $\sigma^m$ is the identity on $k_0$.  Then we may replace $x$ by $x^m$ and $\sigma$ by $\sigma^m$ and assume that the fixed field of $\sigma$ is algebraically closed in $K$ and thus $K$ is a regular extension of its fixed field.  Thus we can pick a $\sigma$-invariant local subring $S$ of $K$ satisfying the conditions of Lemma \ref{lem:S}.  Then $S[x;\sigma]$ surjects onto $(S/P)[x;\sigma]$ and the latter ring has the property that ${\rm Frac}((S/P)[x;\sigma])$ contains pro-unipotent elements that generate a free non-cyclic multiplicative subgroup by Proposition \ref{prop:pc}, and so ${\rm Frac}(S[x;\sigma])$ has this property too by Lemma \ref{lem:reduce}.  Thus we get the result in this case.
\end{proof}
\section{Proof of Theorems \ref{thm: main2} and \ref{thm: main1}}
\label{sec:main}
We now prove the main result of this paper.  We begin by completing the proof of the automorphism case of Theorem \ref{thm: main1}.

\begin{proof}[Proof of Theorem \ref{thm: main2}]
We let $k_0$ denote the prime subfield of $k$ and we pick $c\in K$ such that $\sigma(c)\neq c$.  Then let $K'=k_0(\{\sigma^n(c)\colon n\in \mathbb{Z}\})$.  Then $\sigma$ induces a non-identity automorphism of $K'$ and since $K'(t;\sigma)$ embeds in $K(t;\sigma)$ it suffices to prove the result for $K'(t;\sigma)$.  

If $K'$ is not finitely generated as an extension of $k_0$ then at least one of $K_1:=k_0(\{\sigma^n(c)\colon n\ge 0\})$ and $K_2:=k_0(\{\sigma^n(c)\colon n\le 0\})$ must be non-finitely generated as a field extension of $k_0$.  Without loss of generality, we may assume that $K_1$ is not finitely generated as an extension of $k_0$.  Then by \cite[Proposition 3.2]{BR12}\footnote{We note that the argument in this paper was actually supplied by the anonymous referee.} we have that the elements 
$cx$ and $cx^2$ generate a free $k_0$-subalgebra of $K(x;\sigma)$, and hence $a:=c\sigma(c) x^3 = (cx)(cx^2)$ and $b:=c\sigma^2(c)x^3=(cx^2)(cx)$ generate a free $k_0$-subalgebra of $K(x;\sigma)$.  But by a result of S\'anchez \cite[Theorem pp. 2--3]{Sa} we then have that $1+a$ and $1+b$ generate a free non-cyclic subgroup of $K(x;\sigma)^*$ and these elements are evidently pro-unipotent.  Thus we may assume that $K'$ is a finitely generated extension of $k_0$ and so it suffices to consider the case when $K$ is a finitely generated extension of its prime subfield.
If $k_0$ has characteristic zero, then we get the result from Proposition \ref{prop:zc}; if $k_0$ has positive characteristic then the result follows from Proposition \ref{prop:pc}.
The result now follows.
\end{proof}
 
 We can now immediately deduce Thoerem \ref{thm: main1}.
\begin{proof}[Proof of Theorem \ref{thm: main1}] A folklore result says that $K(x;\sigma,\delta)$ is isomorphic to a division ring of the form $K(x;\sigma)$ or $K(x;\delta)$ with $\sigma$ an automorphism and $\delta$ a derivation (see \cite[Lemma 5.1]{BR12} for a proof).  In the case when $D=K(x;\delta)$ and $\delta\neq 0$, there is some $c\in K$ such that $b:=\delta(c)\neq 0$.  Then $[b^{-1}x,c]=1$ and so $D$ contains a copy of the Weyl division algebra.  It is known that the Weyl division algebra contains a free non-cyclic multiplicative subgroup regardless of the characteristic of the base field and so we get the result in this case.  The automorphism case follows from Theorem \ref{thm: main2}.
\end{proof}

 \section{Application to group algebras}
 \label{sec:app}
In this section, we give as an application a result about the existence of free subgroups in group algebras of solvable groups.  We note that when $G$ is a torsion-free solvable group, we have that the group algebra $k[G]$ is a domain and has a quotient division ring ${\rm Frac}(k[G])$ (see \cite{Pass, KLM}). 
\begin{proof}[Proof of Corollary \ref{corollary:main}]
 We first consider the case when $G$ is a torsion-free non-abelian solvable group.  To show that ${\rm Frac}(k[G])$ contains a free group of rank two, first observe that $G$ contains a non-abelian metabelian subgroup $H$ and since ${\rm Frac}(k[H])$ embeds in ${\rm Frac}(k[G])$, it is enough to consider the case when $G$ is metabelian.  We pick $a,b\in G$ that do not commute and since ${\rm Frac}(k[\langle a,b\rangle])$ embeds in ${\rm Frac}(k[G])$, we can further assume that $G$ is finitely generated.  Let $A=G'$ and pick $x\in G$ that does not commute with every element of the normal subgroup $A$.  If $x^m\in A$ for some $m\ge 1$, then letting $B=\langle A,x\rangle$, we see that $B$ is a non-abelian group that is abelian-by-finite and so ${\rm Frac}(k[B])$ is finite-dimensional over its centre and hence contains a free group of rank two by \cite{Li2, Gon} and thus ${\rm Frac}(k[G])$ contains a free group of rank two.  Thus we may assume that $x^m\not\in A$ for $m\ge 1$.  Thus $E:=\langle A,x\rangle \cong A\rtimes \mathbb{Z}$ and so $k[E]\cong k[A][t,t^{-1};\sigma]$, where $\sigma$ is the non-identity automorphism of $k[A]$ induced by conjugation by $x$.  Then ${\rm Frac}(k[E])$ contains a free group of rank two by Theorem \ref{thm: main1} and so ${\rm Frac}(k[G])$ does since by construction it contains ${\rm Frac}(k[E])$.  The result follows in this case. 
 
 When $G$ is a non-abelian solvable-by-finite, either $G$ is has a non-abelian finite-index solvable subgroup in which case we get the result from the above, or $G$ is abelian-by-finite in which case $k[G]$ is a finite module over its centre and we get the result from \cite{Li2, Gon}.
\end{proof}
%[GD66]
%[Har77] [Mat86] [Sha94]
%, E?l?ements de g?eom?etrie alg?ebrique: IV. E?tude locale des sch?emas et des morphismes de sch?emas, troisi`eme partie, Publications math?ematiques de l?I.H.E?.S., vol. 28, Institut des Hautes E?tudes Scientifiques, 1966.

\end{document}